\documentclass[11pt]{article}
\usepackage{amsmath,amsfonts,amsthm,amscd,amssymb,graphicx}
\numberwithin{equation}{section}

\newtheorem{theorem}{Theorem}[section]

\newtheorem{lemma}[theorem]{Lemma}

\newtheorem{proposition}[theorem]{Proposition}

\newtheorem{remark}[theorem]{Remark}

\newcommand{\RR}{\mathbb{R}}

\def\dt{\partial_t}
\def\dtt{\partial_{\tilde t}}

\def\dx{\partial_x}

\def\dz{\partial_z}
\def\dX{{\partial_X}}
\def\dXt{{\partial_{\tilde X}}}

\def\dZ{{\partial_Z}}
\def\D{\partial}

\def\e{\epsilon}

\def\tV{\tilde V}

\def\tw{\tilde w}

\newcommand{\vv}{{\bf v}}
\newcommand{\xx}{{\bf x}}

\newcommand{\barr}{\begin{array}}
\newcommand{\earr}{\end{array}}
\newcommand{\bmat}{\begin{pmatrix}}
\newcommand{\emat}{\end{pmatrix}}

\newcommand{\Div}{\mbox{div}}

\def\ind{\noindent}

\begin{document}

\title{Boundary layers interactions in the plane parallel incompressible flows}

\author{Toan Nguyen{\footnote{Division of Applied Mathematics, Brown University, 182 George street, Providence, RI 02912, USA. Email: Toan\underline{~}Nguyen@Brown.edu}} \and Franck Sueur{\footnote{Laboratoire Jacques-Louis Lions, Universit\'e Pierre et Marie Curie - Paris 6, 4 Place Jussieu, 75005 Paris, FRANCE. Email: fsueur@ann.jussieu.fr
}}\\
}

\date{July 28, 2011}

\maketitle

\begin{abstract}
We study the inviscid limit problem of the incompressible flows in the presence of both impermeable regular boundaries and a hypersurface transversal to the boundary across which the inviscid flow has a discontinuity jump. In the former case, boundary layers have been introduced by Prandtl as correctors near the boundary between the inviscid and viscous flows. In the latter case, the viscosity smoothes out the discontinuity jump by creating a transition layer which has the same amplitude and thickness as the Prandtl layer. In the neighborhood of the intersection of the impermeable boundary and of the hypersurface, interactions between the boundary and the transition layers must then be considered. In this paper, we initiate a mathematical study of this interaction and carry out a strong convergence in the inviscid limit for the case of the plane parallel flows introduced by Di Perna and Majda in \cite{DM}. 
\end{abstract}


\section{Introduction}

In this paper we are interested in the behavior of the incompressible  Navier-Stokes flow when the viscosity is small. 
This so-called inviscid limit problem is particularly difficult when the flows is contained in a domain limited by impermeable walls. 
In the standard case of a half-space, the problem reads as follows: 
\begin{equation}\label{3dNS} \begin{aligned}\dt \vv^\e + (\vv^\e \cdot \nabla) \vv^\e + \nabla p^\e &= \epsilon \Delta \vv^\e  \\
\Div \; \vv^\e &=0.
\end{aligned}
\end{equation}
Here, $\xx = (x,y,z)$ is  in $ \RR \times \RR \times (0,+\infty)$, the velocity $\vv^\e=(u^\e,v^\e,w^\e)$ is in $ \RR^3$, $p^\e$ the pressure and  $\e >0$ is the viscosity parameter. 

The equation \eqref{3dNS} is imposed with the classical no-slip boundary condition:
\begin{equation}\label{BCs} \vv^\e_{\vert _{z=0}}  = 0.\end{equation}

\bigskip
\ind
Considering the problem \eqref{3dNS}-\eqref{BCs} in the limit $\e\to 0$, 
one may hope to recover the Euler flow: the equation \eqref{3dNS} with $\e=0$, for which the natural condition on the boundary $\{z=0\}$ is 
\begin{equation}\label{BCe} w^0_{\vert _{z=0}}  = 0.\end{equation}

 Due to the difference (or rather, loss) of boundary conditions, it is common in the limit to add a boundary corrector or the so-called  Prandtl layer. This formal procedure was introduced by Prandtl in 1904, and it remains a challenging mathematical problem to circumvent the validity of this theory.  

\bigskip
\ind

Yet, some positive answers have been given in the setting of analytic flows in two dimensions by Caflish and Sammartino in \cite{CS} and improved by Cannonne, Lombardo and Sammartino  in \cite{CLS} for  inviscid flows that are analytic with respect to the tangential variables. On the other hand, when the smoothness with respect to the tangential variables is limited, 
the Prandtl layer have been shown to be unstable; see for example the papers by Grenier \cite{Gre:2000}, G\'erard-Varet and Dormy \cite{GD}, Guo and Nguyen \cite{GN}. Note that these papers also concern the $2$d case.

\bigskip
\ind 

Here, we propose to study the inviscid limit problem of a viscous incompressible  Navier-Stokes flow in presence of both a solid boundary and of a transversal discontinuity hypersurface in the limiting inviscid flow. 
The full problem is currently out of reach.
In particular, jump discontinuity across a hypersurface is also a rather unstable pattern for the incompressible Euler equations, because of the  Kelvin-Helmhotz instabilities. 
Nevertheless, when the inviscid theory is successful to provide some Euler solutions with some  jump discontinuity across a hypersurface, it is expected that the extra viscosity in the Navier-Stokes solutions smoothes out the discontinuity into a transition layer  which can be basically thought as a transmission version of the Prandtl layers. Here again, positive results are known  in an analytic framework, in $2$d, 
 see \cite{WASCOM}. 

\bigskip
\ind 
 
However, since this hypersurface is assumed here to be  transverse to the boundary, one cannot relies on  the previous results  based upon the analyticity in the transversal variables to study the interactions between the boundary layer and the transition layer. We will therefore study the layers interactions in Sobolev spaces.

\bigskip
\ind 

In this paper, we will restrict our study to {\em a  simple setting of three-dimensional incompressible flows: the plane-parallel flows}.
 They were introduced by DiPerna and Majda in \cite{DM} in order  to prove that the Euler equations are not closed under  weak limits  (in three spatial dimensions).
  These flows have also been used  as basic flows for the Euler equations by  Yudovich in \cite{yudo} to investigate stability issues 
   and  recently by  Bardos and Titi in \cite{BT} to investigate several longstanding questions including the minimal regularity needed to have well-posedness results, localization of vortex sheets on surfaces, and the energy conservation for the Euler equations.
   
  \bigskip
\ind 

Precisely, a plane-parallel solution is of the form:
\begin{equation}\label{PP}
\vv^\e (t,x,y,z) = \begin{pmatrix}u^\e(t,z)\\ v^\e(t,x,z)\\0\end{pmatrix} .
\end{equation}
Then, the Navier-Stokes system \eqref{3dNS} depletes into 
\begin{equation}\label{plane-NSeqs} 
\begin{aligned}\dt u^\e &= \e \dz^2 u^\e 
\\ \dt v^\e + u^\e \dx v^\e &= \e \Delta_{xz} v^\e ,
\end{aligned}
\end{equation}
with $p^{\e} =0$. It is thus a pressureless flow. Observe that a vector field  of the form  \eqref{PP} is divergence free.
On the other hand  the boundary conditions \eqref{BCs}  now read  
\begin{equation}\label{plane-NScds}
(u^\e,v^\e)_{\vert_{z=0}}=0,
\end{equation}
as the Dirichlet condition for the third component  is automatically satisfied for flows of the form \eqref{PP}.
The system \eqref{plane-NSeqs}-\eqref{plane-NScds} is now quite simple: the first equation in  \eqref{plane-NSeqs} is a one dimensional  heat equation whereas the second one is a two dimensional transport-diffusion equation, and for both we prescribe homogeneous Dirichlet conditions.

\bigskip
\ind

On the other hand, the Euler system, the equations \eqref{3dNS}  with  $\e=0$, depletes into 
\begin{equation}\label{plane-Euler} \begin{aligned}\dt u^0 &= 0
\\ \dt v^0 + u^0 \dx v^0 &= 0.
\end{aligned}
\end{equation}
Therefore the solution starting from the initial data 
\begin{equation*}
\vv_0 (x,z) = \begin{pmatrix}u_0 (z)\\ v_0 (x,z)\\0\end{pmatrix} 
\end{equation*}
is simply given by the formula 
\begin{equation}\label{PP^0}
\vv^0 (t,x,z) = \begin{pmatrix}u_0 (z) \\ v_0 (x - t u_0 (z) ,z) \\0\end{pmatrix} .
\end{equation}
This holds true in a quite general setting, but let us be formal for a few more lines.
For instance let us think that the function $\vv_0 $ is smooth for a while, so that there is no doubt to have about the meaning of the formula  \eqref{PP^0} nor about the fact that it solves the depleted Euler equations \eqref{plane-Euler}. 
We want to focus here first on the issue of the boundary conditions. 
 In particular, note that no boundary conditions are needed to be prescribed for the system \eqref{plane-Euler}, since any solution of the form 
 \eqref{PP} already satisfies the condition \eqref{BCe}. 
 On the other hand, if the initial data $\vv_0$ does not vanish on the boundary $z=0$ then neither does the corresponding solution $\vv^0$ given by  \eqref{PP^0} for positive times. 
 As a consequence, $\vv^0$ does not satisfy the condition \eqref{plane-NScds} and therefore  cannot be a good approximation, say in $L^\infty$, of any smooth solution  $\vv^\e$ of the system \eqref{plane-NSeqs}-\eqref{plane-NScds}.
 Yet Prandtl's theory predicts that the system \eqref{plane-NSeqs}-\eqref{plane-NScds} admit some solutions $\vv^\e$ which have the following asymptotic expansion as $\e\to 0$:
 \begin{equation}\label{exp-Pr} \vv^\e(t,x,z) \sim \vv^0(t,x,z) + \vv_P(t,x,\frac{z}{\sqrt \e}).\end{equation}
Above the profile $\vv_P (t,x,Z)$ describes a Prandtl boundary layer correction. In particular it satisfies $\vv_P (t,x,Z) \rightarrow 0$ when $Z \rightarrow +\infty$, so that this term really matters only in a layer of thickness $\sqrt \e$ near the boundary $\{z=0\}$,  and  also satisfies 
 $ \vv^0 (t,x,0) + \vv_P (t,x,0) = 0$, so that the functions in the right hand side of \eqref{exp-Pr} satisfies the boundary conditions \eqref{plane-NScds}.
The validity of this  asymptotic expansion has been verified in a recent paper of Mazzucato, Niu and Wang \cite{Maz} for regular initial data $ \vv_0$. 
In particular it follows easily from their analysis that for any  regular initial data $ \vv_0$, there exists a sequence of 
smooth solutions $\vv^\e$  of the system \eqref{plane-NSeqs}-\eqref{plane-NScds}, with some initial data conveniently chosen, such that $\vv^\e$ converges to $\vv^0$  strongly in the $L^2$ topology.

\bigskip
\ind

Here, as mentioned previously,  we are interested in the case where $\vv_0$ has a jump of discontinuity across a hypersurface. 
More precisely we assume that $u_0$ is smooth and that  $v_0$  is piecewise smooth with a jump of discontinuity across the hypersurface $\{ x = 0\}$:
\begin{equation}\label{jump-assmp}[v_0 ] _{\vert_{x=0}} := \lim_{x\to 0^+}v_0(x,z) - \lim_{x\to 0^-}v_0(x,z) \not =0.\end{equation}
We assume for simplicity that there is no jump of the normal derivative of $v_0$  across the hypersurface $\{ x = 0\}$, that is 
\begin{equation}\label{jump-assmp2}  [\dx v_0 ] _{\vert_{x=0}} =0.\end{equation}

Then it can be easily seen on the formula \eqref{PP^0} that the corresponding Euler solution $\vv^0$ is  piecewise smooth with a jump of discontinuity across the hypersurface given by the equation $\{ \Psi^0 (t,x,z) = 0\}$, with
 $$\Psi^0(t,x,z) := x -\psi (t,z), \qquad  \psi(t,z):= t u_0 (z) ,$$
Moreover taking the derivative with respect to $x$ of the  both sides of Formula  \eqref{PP^0} yields that  there is no jump of the normal derivative of $v_0$  across the hypersurface $\{ \Psi^0 (t,x,z) = 0\}$.

Such a pattern cannot hold anymore for any reasonable solutions  $\vv^\e$  of the depleted Navier-Stokes equations \eqref{plane-NSeqs}-\eqref{plane-NScds}: 
the viscosity  smoothes out  this jump of discontinuity into a transition layer 
  near the hypersurface $\{\Psi^0 = 0\}$. 
  In particular  $\vv^\e$  and its normal derivative must be continuous across the hypersurface $\{ \Psi^0  = 0\}$:
\begin{equation}\label{nojump}
[ \vv^\e ] _{\vert_{\Psi^0 = 0}}  =0 \text{ and }  [\dx  \vv^\e  ] _{\vert_{\Psi^0 = 0}} =0 .
\end{equation}
Following Prandtl's ideas it is natural to introduce a corrector 
\begin{equation*}
 \vv_{KH}(t,\frac{\Psi^0(t,x,z)}{\sqrt \e},z)
\end{equation*} 
where the profile $V_{KH} (t,x,X)$ satisfies\footnote{Here the subscript $KH$ holds for Kelvin-Helmhotz} 
\begin{equation}\label{eqs-vkhIntro}
\left\{ \begin{array}{cccc} 
&[V_{KH}]_{\vert_{X=0}} &=& -[v_0]_{x=0}
\\
&[\dX V_{KH}]_{\vert_{X=0}} &=&0,
\end{array}\right.\end{equation}
and $V_{KH}\to 0$ as $X\to \pm\infty$. 
This strategy can be seen as a transmission counterpart of the introduction of the boundary layer $\vv_P$ previously mentionned.
Actually, if the fluid domain was not limited by the boundary $\{ z = 0\}$ one could then adapt the analysis of \cite{Maz} to justify the existence of some solutions  $\vv^\e$ of \eqref{plane-NSeqs}-\eqref{plane-NScds} which admits  an expansion of the form 
\begin{equation*}
\vv^\e(t,x,z) \sim \vv^0(t,x,z) + \vv_{KH}(t,\frac{\Psi^0(t,x,z)}{\sqrt \e},z).
\end{equation*} 

Yet there is no reason for which the transition layer $\vv_{KH}$ should satisfy the boundary condition at $z=0$, nor for which the boundary layer $\vv_P$ should take care of the jump condition across $\{\Psi^0=0\}$. 
It is precisely our point to understand how to deal with both layers.

\bigskip
\ind

Our result is the following. 

\begin{theorem}\label{theo-conv} 
Let $1<p<2$ and let 
\begin{equation}\label{initial-data}
\begin{aligned}
u_0 (z) &\in H^{2} (0,+\infty)  , \\ v_{0,+} (x,z) & \in  W^{2,p} ([0,+\infty) \times (0,+\infty) ) ,  \\  v_{0,-}  (x,z) &\in  W^{2,p} ((-\infty ,0 ] \times (0,+\infty)  ) ,
\end{aligned}\end{equation}
 and  
\begin{equation}\label{def-v0}
v_{0} \in L^{p} ( \RR \times   (0,+\infty)), \qquad \mbox{with}\quad v_{0}(x,z) := \left\{ \begin{array}{lcl} v_{0,+} (x,z), \qquad & x>0 ,\\
  v_{0,-} (x,z), \qquad &x<0.
\end{array}\right. \end{equation}
Assume that $\vv_0$ satisfies the jump conditions \eqref{jump-assmp} and \eqref{jump-assmp2}. 
Let us consider  $\vv^0$ given by  the formula \eqref{PP^0}, which for any  $T>0$  is a distributional solution of  the depleted Euler equations  \eqref{plane-Euler} with $\vv^0_{\vert_{t=0}} = \vv_0$.

Then, there exist some smooth solutions $\vv^\e := (u^\e(t,z) , v^\e(t,x,z) )$ of the depleted Navier-Stokes equations \eqref{plane-NSeqs}-\eqref{plane-NScds} such that 
as $\e \to 0$, there holds the convergence
\begin{equation}\label{eqs-conv}\begin{aligned}
&\vv^\e \to \vv^0 \qquad \mbox{in } L^\infty(0,T;  L^{2} (\RR_+)   \times L^{p} ( \RR \times \RR_+)) .
\end{aligned}
\end{equation} 
\end{theorem}
\bigskip

Here $W^{2,p} $ denotes the usual Sobolev space of order $2$ associated to the Lebesgue space $ L^{p}$ and $H^{2}$ denotes the special case $H^{2} := W^{2,2} $.

\bigskip
Let us end our Introduction by giving here a few comments.

First, observe that in the statement of Theorem \ref{theo-conv} the initial data of  $\vv^\e$ is not prescribed. In the proof, we will explicitly choose them in a convenient way; in particular, it allows the boundary and transmission layers to be initially specified.
This could perhaps seem a little bit unusual at first glance, but it is in fact only technical for our convenient formulation of the main result. 
However this way to formulate our results avoids some  extra considerations regarding forcing terms and/or initial layers which do not seem essential for our purpose in the present paper. 

Finally, let us mention that we are unable to include the case $p=2$ or any $p>2$ in Theorem \ref{theo-conv}.  We will explain why in Remark \ref{notL2}.

\section{Straightened interface}

  To fix the interface, we introduce the following change of variable:
$$ \tilde x := x - \psi(t,z), \qquad   \text{ where }\psi(t,z):= t u_0 (z).$$ 
In these coordinates, the discontinuity interface is given by the equation $\tilde x=0$. 
In what follows, we drop the tilde in $\tilde x$. 
The system \eqref{plane-NSeqs} now reads
\begin{equation}\label{plane-NSE} \begin{aligned}\dt u^\e &= \e \dz^2 u^\e ,
\\
\dt v^\e + (u^\e - u_0) \dx v^\e &= \e  \Delta^\psi_{xz} v^\e  ,
\end{aligned}
\end{equation}
with 
$$ \Delta^\psi_{xz}: = \dx^2 + (\dz -\dz\psi \dx)^2 = (1+|\dz\psi|^2 )\dx^2- 2\dz\psi \D_{z,x}^2 - \dz^2 \psi\dx + \dz^2.$$

The boundary conditions \eqref{plane-NScds} do not change:
\begin{equation}\label{plane-NScds2}
(u^\e,v^\e)_{\vert_{z=0}}=0 .
\end{equation}
We are looking for some functions $ u^\e $ and  $ v^\e $  which satisfy the equations \eqref{plane-NSE} on both quadrants 
$(x,z) \in  (0,+\infty) \times (0,+\infty)$
and
$(x,z) \in  (-\infty ,0 ) \times (0,+\infty)$
with the interface conditions:
\begin{equation}\label{nojump2}
[ \vv^\e ] _{\vert_{x = 0}}  =0 \text{ and }  [\dx  \vv^\e  ] _{\vert_{x = 0}} =0 ,
\end{equation}
which correspond to the conditions \eqref{nojump} in the new variables.
Now, since $u^\e$ does not depend on $x$, the conditions \eqref{nojump2} reduce to 
\begin{equation}\label{nojump3}
[ v^\e ] _{\vert_{x = 0}}  =0 \text{ and }  [\dx  v^\e  ] _{\vert_{x = 0}} =0 .
\end{equation}
Note that if  $ u^\e $ and  $ v^\e $   are distributional solutions of \eqref{plane-NSE}  on both quadrants $(x,z) \in  (0,+\infty) \times (0,+\infty)$
and
$(x,z) \in  (-\infty ,0 ) \times (0,+\infty)$
 and satisfy the previous interface conditions then they are distributional solutions of \eqref{plane-NSE}  on  the whole half-space $ \RR \times  (0,+\infty)$.

In the limit case $\e =0$, the situation is now particularly simple: in the new coordinates the solution $\vv^0$ is stationary 
\begin{equation}\label{PP0}
\vv^0 (t,x,z) = \begin{pmatrix}u_0 (z)\\ v_0 (x,z)\\0\end{pmatrix} .
\end{equation}

Now, to prove Theorem \ref{theo-conv} it suffices to prove that there exist  some functions $ u^\e $ and  $ v^\e $  which satisfy the equations \eqref{plane-NSE} on both quadrants, satisfy the conditions \eqref{plane-NScds2} and \eqref{nojump3} and converge, 
as $\e \to 0$,  to $\vv^0$ given by \eqref{PP0} in $  L^\infty(0,T;  L^{2} (\RR_+)   \times L^{p} ( \RR \times \RR_+))$.

\section{Asymptotic expansions}

Let us now describe our strategy.
We are going to construct a family of functions of the form
\begin{equation}\label{appx-solns}
\begin{aligned}
u^\e_{app} (t,z) &= u_{0} (z) + U_P (t,\frac{z}{\sqrt \e}),\\
v^\e_{app}(t,x,z) &= v_{0} (x,z) + V_P(t,x,\frac z{\sqrt\e}) +V_{KH}(t,\frac{x}{\sqrt \e},z)
+V_{b} (t,\frac{x}{\sqrt \e}, \frac {z}{\sqrt \e}) ,
\end{aligned}
\end{equation} 
which satisfy  approximatively \eqref{plane-NSE}  on both quadrants (in a sense that we will precise in the sequel), and which satisfy the conditions  \eqref{plane-NScds2}  and  \eqref{nojump3}.

In \eqref{appx-solns}, $ (u_{0},v_{0})$ are the functions given by \eqref{PP0}.
The other functions will be defined in the sequel. 
For instance, $(U_P,V_P)$ will be a depleted Prandtl layer near the boundary, $V_{KH}$ will be a transmission layer near the discontinuity interface, and $V_b$ will aim at describing the behavior of the boundary layers interaction. 
In what follows, we will use, as in the introduction, the capitalized variables $X,Z$ to refer to $x/{\sqrt\e},z/{\sqrt\e}$, correspondingly.

Then we will prove that there exists a family of functions  $(u^\e , v^\e )$ close to  $(u^\e_{app} , v^\e_{app} )$ which  exactly satisfy \eqref{plane-NSE}  on both quadrants
and
$(x,z) \in  (-\infty ,0 ] \times (0,+\infty)$, and which still satisfy the conditions  \eqref{plane-NScds2}  and  \eqref{nojump3}.

It will remain to prove that this family $(u^\e , v^\e )$  converges to $(u_0 , v_0 )$  in $L^\infty(0,T;  L^{2} (\RR_+)   \times L^{p} ( \RR \times \RR_+))$, for $1< p<2$, to conclude the proof of  Theorem \ref{theo-conv}.

\subsection{Construction of the approximated solution}
\label{consapp}

\subsubsection{Construction of $U_{P}$}

We start by defining the function $U_{P}$, which aims at compensating the non-vanishing value of $ u_{0}$ at $z=0$. On the other hand, we want this correction to be localized near  the boundary $z=0$. We will therefore require  $U_{P}$ to satisfy 
\begin{equation}
\label{mal1}
U_{P}(t,0) = - u_{ 0}(0), \qquad \lim_{Z\to +\infty} U_P (t,Z) = 0.
\end{equation}
Now if we  put the Ansatz \eqref{appx-solns} into the system \eqref{plane-NSE} and match the order in $\e$, we get from the equation for $u^\e$ the following equation for the profile  $U_P (t,Z)$:
\begin{equation}\label{eqs-PrUp}
  \dt U_P = \partial^2_Z U_P .
  \end{equation}
We choose for $U_{P}$ the initial value 
\begin{equation}
\label{mal2}
{U_{P}} _{ \vert_{t=0}} (Z)= -  u_{ 0} (0) e^{-Z} .
\end{equation}
By  Duhamel's principle, the solution $U_p(t,Z)$ of \eqref{mal1}-\eqref{eqs-PrUp}-\eqref{mal2} satisfies 
\begin{equation*}
\begin{aligned}
U_{p} (t,Z) =& 
-  u_{ 0}(0) e^{-Z} - u_0(0) \int_{0}^{t}\int_{0}^{+\infty} \mathcal{G}(t-s,Z;Z') e^{-Z'} dZ' ds,
\end{aligned}
\end{equation*}
where $\mathcal{G}(t,Z;Z')$ denotes the one-dimensional heat kernel on the half-line:
\begin{equation}\label{heat-kernel}
\mathcal{G}(t,Z;Z') := G(t,Z-Z') - G(t,Z+Z'), \qquad G (t,Z) := \frac{1}{\sqrt{4\pi t}} e^{-\frac{Z^2}{4 t} }.
\end{equation}
Now by using the standard convolution inequality: $\|f * g\|_{L^p} \le \|f\|_{L^p} \|g\|_{L^1}$, we easily deduce that, for any $p>1$, 
\begin{equation*} 
 \| U_P  \|_{  L^\infty (0,T ; L^{p} (\RR_+ ))} \quad\le\quad C_0  | u_{ 0} (0)  | \Big[1+ \int_0^T \| G(t,\cdot)\|_{L^1} \|e^{-Z}\|_{L^p}\; ds \Big] \quad\le\quad C_0 |u_0(0)|,
\end{equation*}
for some positive constant $C_0$ that depends on $p$ and $T$. Here, we used the fact that $\| G(t,\cdot)\|_{L^1} = 1$. Similarly, using the fact that $\| \partial_ZG(t,\cdot)\|_{L^1} \sim t^{-1/2}$, we obtain
\begin{equation*}
 \| \partial_ZU_P  \|_{  L^\infty (0,T ; L^p (\RR_+ ))} \quad\le\quad C_0  | u_{ 0} (0)  |  \Big[1+ \sup_{0\le t\le T}\int_0^t  (t-s)^{-1/2} ds \Big],
\end{equation*}
which is again bounded by $ C_0 |u_0(0)|$. 

That is, we obtain the following  lemma:
 \begin{lemma} \label{lem-Up} There exists a unique solution $U_p$ to the problem \eqref{mal1}-\eqref{eqs-PrUp}-\eqref{mal2} on $[0,T]\times \RR_+$, for any $T>0$. Furthermore, for any $p>1$,  there is some positive constant $C_0$ that depends on $p$ and $T$ such that 
 \begin{equation}\label{est-PrUp}
 \|U_P  \|_{  L^\infty (0,T ; W^{1,p} (\RR_+ ))} \quad\le\quad C_0  | u_{ 0} (0)  |.
\end{equation}
\end{lemma}

\subsubsection{Construction of $V_{P}$}

For $V_{P}$, the situation is the same as that for $U_{P}$, other than the fact that $V_{P}$ also depends on the variable $x$. However, $x$ only appears as  a harmless parameter. More precisely, by plugging the Ansatz \eqref{appx-solns} into the system \eqref{plane-NSE} and match the order in $\e$, we then get from the equation for $v^\e$ the profile equation for $V_P (t,x,Z)$:
\begin{equation}
\label{eqs-PrVp}
\dt V_P = \partial^2_Z V_P,\qquad {V_P}_{\vert_{Z=0}} = -{v_0(x,z)}_{\vert_{z=0}}, \qquad \lim_{Z\to +\infty} V_P = 0.
\end{equation}
Once again, we choose an initial data compatible with the boundary condition, for instance
\begin{equation}
\label{eqs-PrVpInit}
{V_P}_{\vert_{t =0}} =  - v_0 (x,0) e^{-Z} .
\end{equation}
Then as was the case for $U_p$, there exists a unique solution $V_P$  of \eqref{eqs-PrVp}-\eqref{eqs-PrVpInit} satisfying the Duhamel principle:
 \begin{equation*}
\begin{aligned}
V_{p} (t,x,Z) =& 
-  v_{ 0}(x,0) e^{-Z} - v_0(x,0) \int_{0}^{t}\int_{0}^{+\infty} \mathcal{G}(t-s,Z;Z') e^{-Z'} dZ' ds,
\end{aligned}
\end{equation*} where $\mathcal{G}(t-s,Z;Z')$ is the heat kernel defined as in \eqref{heat-kernel}. 
It is clear from this integral representation for $V_p$ that the only dependence on $x$ is due to $v_0(x,0)$.  Thus, we easily obtain the following lemma. 
 \begin{lemma} \label{lem-Vp} There exists a unique solution $V_p$ to the problem \eqref{eqs-PrVp}-\eqref{eqs-PrVpInit} on $[0,T]\times \RR_\pm\times \RR_+$, for any $T>0$. Furthermore, for any $p>1$,  there is some positive constant $C_0$ that depends on $p$ and $T$ such that 
 \begin{equation}\label{est-PrVp}
 \| \partial_x^k V_P\|_{L^\infty (0,T ; L^p (\RR_\pm ; W^{1,p} (\RR_+ )))} \quad\le\quad C_0  \| \partial_x^k v_{ 0} (\cdot,0)  \|_{L^p (\RR_\pm)}, \qquad k = 0,1,2, 
\end{equation}
and the jump of discontinuity $[V_P]_{\vert_{x=0}}$ satisfies \begin{equation}\label{est-PrVp-jump}
 \| [V_P]_{\vert_{x=0}}\|_{  L^\infty (0,T ; W^{1,p} (\RR_+ ))} \quad\le\quad C_0  | [v_{ 0} (x,0)]_{\vert_{x=0}}  |. 
\end{equation}
\end{lemma}
\begin{proof} Similarly as done for $U_p$, the integral representation for $V_p$ easily yields
\begin{equation}\label{est-PrVp-x}
 \| \partial_x^k V_P (x)\|_{  L^\infty (0,T ; W^{1,p} (\RR_+ ))} \quad\le\quad C_0  | \partial_x^k v_{ 0} (x,0)  |, \qquad k = 0,1,2, 
\end{equation}
for each nonzero $x \in \RR$. Taking the $L^p$ norm of this inequality in $x$ gives \eqref{est-PrVp} at once. The estimate for the jump of discontinuity of $V_p$ follows similarly by noting that the jump $[V_p]_{\vert_{x=0}}$ satisfies the similar integral representation to that of $V_p$.

\end{proof}

\subsubsection{Construction of $V_{KH}$}
Similarly, plugging the Ansatz \eqref{appx-solns} into the system \eqref{plane-NSE} yields the profile equation for $V_{KH}$:
\begin{equation}
\label{eqs-vkh}
\dt V_{KH}  = (1+|\dz\psi|^2) \partial_X^2  V_{KH}
,\qquad [V_{KH}]_{\vert_{X=0}} = - [v_0]_{x=0},
\qquad [\dX V_{KH}]_{\vert_{X=0}} =0.
\end{equation}
We choose the initial data:
\begin{equation}
\label{eqs-vkh-INIT}
V_{KH} {\vert_{t=0}} (X,z) =  \mp \frac{[v_0]_{x=0}}{2}  e^{\mp X} , \qquad  \pm X>0 .
\end{equation}
We will derive necessary estimates for the profile $V_{KH}$. It turns out convenient to introduce a change of variables:
$$ \tilde X = X, \qquad \tilde t = \int_0^t (1+|\dz \psi(s,z)|^2) \; ds,$$ and write $$ V_{KH}(t,X,z) = \tV_{KH}(\tilde t(t,X,z),\tilde X(t,X,z), z).$$ 
In these new variables, we then have 
\begin{equation}
\label{eqs-tvkh}
\dtt \tV_{KH}  = \partial_{\tilde X}^2  \tV_{KH}
,\qquad [\tV_{KH}]_{\vert_{\tilde X=0}} = - [v_0]_{x=0},
\qquad [\dXt \tV_{KH}]_{\vert_{\tilde X=0}} =0,
\end{equation}
with initial data
\begin{equation}
\label{eqs-tvkh-INIT}
\tV_{KH} {\vert_{\tilde t=0}} =  \mp \frac{[v_0]_{x=0}}{2}  e^{\mp \tilde X} , \qquad  \pm \tilde X>0 .
\end{equation}
The systems \eqref{eqs-tvkh} and \eqref{eqs-tvkh-INIT} are the heat equations on each half lines $\tilde X<0$ and $\tilde X>0$, with $z$ being a parameter. Thus, the Duhamel principle for the heat equation yields a candidate for $\tV_{KH}(\tilde t,\tilde X,z)$ as 
\begin{equation*}
\begin{aligned}
\tV_{KH} (\tilde t,\tilde X,z) =& 
- \frac{[v_0]_{x=0}}{2}  e^{- \tilde X}  -  \frac{[v_0]_{x=0}}{2}  \int_{0}^{\tilde t}\int_{0}^{+\infty} \mathcal{G}(\tilde t-\tilde s,\tilde X;\tilde X') e^{-\tilde X'} d\tilde X' d\tilde s,
\end{aligned}
\end{equation*}
for $\tilde X>0$, and 
\begin{equation*}
\begin{aligned}
\tV_{KH} (\tilde t,\tilde X,z) =& 
 \frac{[v_0]_{x=0}}{2}  e^{ \tilde X}  +  \frac{[v_0]_{x=0}}{2}  \int_{0}^{\tilde t}\int_{-\infty}^{0} \mathcal{G}(\tilde t-\tilde s,\tilde X;\tilde X') e^{\tilde X'} d\tilde X' d\tilde s,
\end{aligned}
\end{equation*}
with the Green function for the heat equation defined by $$\mathcal{G}(t,X;X')  = G(t,X-X') - G(t,X+X') , \qquad G(t,X) = \frac{1}{\sqrt{4\pi t}}  e^{-X^2/4t}.$$

It is straightforward to check that these definitions of $\tV_{KH}$ on $\RR_\pm \times \RR_+$ indeed satisfy the boundary and jump conditions from  \eqref{eqs-tvkh} and \eqref{eqs-tvkh-INIT}. 
 
Furthermore, similarly to those estimates obtained for $U_p$ and $V_p$, we can easily obtain 
 \begin{equation*}
 \| \dz^k \tV_{KH}(z) \|_{  L^\infty (0,\tilde T ; W^{1,p} (\RR_\pm ))} \quad\le\quad C_0  | [\dz^k v_0(x,z)]_{\vert_{x=0}} |,\qquad, k=0,1,2,
\end{equation*}
for each $z\in \RR_+$ and for some positive constant $C_0$ that depends only on $p$ and $\tilde T$. Going back to the original coordinates $(t,X)$, we have thus shown
 \begin{equation}\label{est-VKH-z}
 \|\dz^k V_{KH}(z) \|_{  L^\infty (0, T ; W^{1,p} (\RR_\pm ))} \quad\le\quad C_0  | [\dz^k v_0(x,z)]_{\vert_{x=0}} |,
\end{equation}
for $k=0,1,2$ and for each $z\in \RR_+$.   

\bigskip

Collecting these information, we obtain the following lemma.  
 \begin{lemma} \label{lem-VKH} There exists a unique solution $V_{KH}$ to the problem \eqref{eqs-vkh} and \eqref{eqs-vkh-INIT} on $[0,T]\times \RR_\pm\times \RR_+$, for any $T>0$. Furthermore, for any $p>1$,  there is some positive constant $C_0$ that depends on $p$ and $T$ such that 
 \begin{equation}\label{est-VKH} 
 \| \dz^k V_{KH}\|_{  L^\infty (0, T ; L^p(\RR_+; W^{1,p} (\RR_\pm )))} \quad\le\quad C_0  \| [\dz^k v_0(x,\cdot)]_{\vert_{x=0}} \|_{L^p(\RR_+)},
\end{equation}
for $k = 0,1,2$.
\end{lemma}
\begin{proof}
The estimate \eqref{est-VKH} is obtained easily by taking to both sides of \eqref{est-VKH-z} the usual $L^p$ norm in $z$ and using the triangle inequality.
\end{proof}

\subsubsection{Construction of $V_{b}$}
Finally, in the ``box'' where the interactions take place, we obtain from the equation for $v^\e$ with the Ansatz \eqref{appx-solns} the following equation for the interaction profile $V_b(t,X,Z)$:
 \begin{equation}\label{eqs-box1} 
\dt V_{b}= \Delta_{XZ}^{\psi_0} V_{b} ,\qquad \Delta^{\psi_0}_{XZ} = \partial_X^2 + (\dZ -\dz\psi_{\vert_{z=0}} \dX)^2,
\end{equation} 
where 
$$\Delta^{\psi_0}_{XZ}: = \partial_X^2 + (\dZ -\dz\psi_{\vert_{z=0}} \dX)^2 ,$$ 
with the boundary and jump conditions:
\begin{equation}
\label{jump-Vb}
{V_{b}}_{\vert_{Z=0}} =- {V_{KH}}_{\vert_{z=0}},\qquad
 [V_{b}]_{\vert_{X=0}} = - [V_P]_{\vert_{x=0}} ,\qquad  [\dX V_{b}]_{\vert_{X=0}} = 0,
\end{equation}
and $V_{b}\to 0$ as $X\to \pm\infty$ or $Z\to +\infty$. Next, we choose the initial data for $V_b$:
\begin{equation}\label{initial-Vb}
{V_{b}}_{\vert_{t=0}} =  \left\{\begin{array}{rll} -  \frac12 e^{- |X|} [V_P]_{\vert_{x=0, t=0}}  , \qquad &X> 0\\
\frac12 e^{- |X|} [V_P]_{\vert_{x=0, t=0}}, \qquad &X< 0
.\end{array}\right.
\end{equation}
which satisfy all the conditions in \eqref{jump-Vb}, thanks to 
\eqref{eqs-PrVpInit}
and to 
\eqref{eqs-vkh-INIT}.

We observe at once that these boundary and jump conditions in \eqref{mal1}, \eqref{eqs-PrVp}, \eqref{eqs-vkh}, and \eqref{jump-Vb} make the Ansatz $\vv^\e_{app}$ defined as in \eqref{appx-solns} smooths out the inviscid solution $\vv^0$ (at least with $C^1$ regularity) as well as satisfy the correct no-slip boundary conditions \eqref{plane-NScds} for the depleted Navier-Stokes system. We will see in the next section that these Ansatz $\vv^\e_{app}$ indeed provide a good approximation for $\vv^\e$, and are sufficient to show the desired convergence.

We will show in this section that the profile $V_b$ exists and we then derive necessary estimates to carry out the convergence stated in the main Theorem \ref{theo-conv}. In fact, we could continue our study by employing the Green function of the heat equation on the half-space as done previously on the half-line. However, we choose to proceed the analysis by energy estimates, as it appears natural for the proof of our desired convergence.

\bigskip
\ind
To begin, it appears convenient to introduce $\tw$ through 
$$V_b =\left\{\begin{array}{rll}&\tw - \frac 12 [V_P]_{\vert_{x=0}} e^{- |X|}, \qquad &X> 0\\
&\tw + \frac 12 [V_P]_{\vert_{x=0}} e^{-| X|}, \qquad &X< 0
.\end{array}\right.$$
The function $\tw$ then solves 
\begin{equation}\label{eqs-tw}\begin{aligned}
\dt \tw &= \Delta^{\psi_0}_{XZ} \tw + J_+ e^{-|X|}, \qquad X > 0,\\
\dt \tw &= \Delta^{\psi_0}_{XZ} \tw + J_- e^{-|X|}, \qquad X < 0,
\end{aligned}\end{equation}
 with boundary and jump conditions: 
\begin{equation}\label{BC-jump-w}[\tw]_{\vert_{X=0}} =  0,\qquad  [\dX\tw]_{\vert_{X=0}} = 0, \qquad \lim_{|X|\to \infty} \tw = \lim_{Z\to \infty} \tw = 0,
\end{equation}
and 
\begin{equation}\label{BC-bdry-w} \left\{ \begin{array}{lll} \tw_{\vert_{Z=0}} =- {V_{KH}}_{\vert_{z=0}} - \frac12 e^{- |X|}[{v_0}_{\vert_{z=0}}]_{\vert_{x=0}}, \qquad &X> 0 
\\\tw_{\vert_{Z=0}} =- {V_{KH}}_{\vert_{z=0}} + \frac12 e^{- |X|}[{v_0}_{\vert_{z=0}}]_{\vert_{x=0}}, \qquad &X<0.\end{array}\right.
\end{equation}
Moreover $\tw $ vanishes at the initial time:
\begin{equation}\label{tw-initial}
\tw {\vert_{t=0}} = 0 .
\end{equation}
Here, $J_\pm$ in \eqref{eqs-tw} collects the terms involving the jumps of discontinuity. Direct calculation together with a use of \eqref{eqs-PrVp} gives
\begin{equation}\label{def-Jpm} \begin{aligned}
J_+ &=  - \frac 12 \Big((1+|\dz\psi_{\vert_{z=0}}|^2)[V_P]_{\vert_{x=0}} + 2 \dz\psi_{\vert_{z=0}} [\dZ V_P]_{\vert_{x=0}} \Big),\\
J_- &=  ~~ \frac 12 \Big((1+|\dz\psi_{\vert_{z=0}}|^2)[V_P]_{\vert_{x=0}} - 2 \dz\psi_{\vert_{z=0}} [\dZ V_P]_{\vert_{x=0}} \Big).
\end{aligned}
\end{equation} 
By applying the estimate \eqref{est-PrVp-jump} obtained in Lemma \ref{lem-Vp}, we then have
\begin{equation}\label{est-Jpm} \int_0^T \|J_\pm(t)\|^p_{W^{1,p}_Z} \; dt \quad\le\quad C_0  | [v_{ 0} (x,0)]_{\vert_{x=0}}|,
\end{equation}
for some $C_0$ that depends on $p$ and $T$. 

\bigskip
\ind
We are able to provide the following estimates. 
\begin{lemma}\label{lem-box-layer} There exists a unique solution $\tw$ to the problem \eqref{eqs-tw}--\eqref{BC-bdry-w} on $[0,T]\times \RR_\pm\times \RR_+$, for any $T>0$. Furthermore, for any $p>1$,  there is some positive constant $C_0$ that depends on $p$ and $T$ such that  there holds
\begin{equation}\label{L2est-box}
\begin{aligned}
\frac {d}{dt} \|\tw \|_{L^p_XW^{1,p}_Z}^p&+\int_{\RR^2_+}| \tw |^{p-2} |\nabla_{X,Z}^{\psi_0}\tw|^2 \; dXdZ+\int_{\RR^2_+}|\dZ \tw |^{p-2} |\nabla_{X,Z}^{\psi_0} \dZ\tw|^2 \; dXdZ
\\& \le \quad C_0\Big(1+ \|J_\pm\|_{W^{1,p}_Z} ^p + \|\tw \|_{L^p_{X}W^{1,p}_Z}^p \Big)
\end{aligned}\end{equation}
Here, $\nabla^{\psi_0}_{XZ}:=(\dX,\dZ-\dz\psi_{\vert_{z=0}}\dX)$.
\end{lemma}
\begin{proof} By multiplying by $|\tw |^{p-2}\tw$ to the equation \eqref{eqs-tw} and integrating it over $\RR^2_+$, one has
$$\begin{aligned}
\frac1p\frac{d}{dt}\|\tw\|_{L^p_{XZ}}^p = \int_{\RR^2_+}\Big(\Delta^{\psi_0}_{XZ} \tw + J_\pm e^{\mp X}\Big)|\tw |^{p-2} \tw \;dXdZ.
\end{aligned}$$
For the term involving $J_\pm$, the standard H\"older's inequality gives 
$$\int_{\RR^2_+}e^{\mp X}J_\pm |\tw |^{p-1}\;dXdZ \quad\lesssim\quad \|\tw \|_{L^p_{XZ}}^{p-1} \|J_\pm\|_{L^p_Z}. $$ 
Here and in what follows, by $f\lesssim g$ we always mean that $f \le C_0 g$, for some positive constant $C_0$ that only depends on $p$ and $T$. 
Now, integration by parts yields 
$$\begin{aligned}
\int_{\RR^2_+}\Delta^{\psi_0}_{XZ}& \tw |\tw |^{p-1}\tw \;dXdZ 
\\=& -(p-1)\int_{\RR^2_+} |\tw|^{p-2}|\nabla^{\psi_0}_{XZ} \tw|^2 \;dXdZ - \int_{\RR}  (\dZ-\dz\psi_{\vert_{z=0}}\dX)\tw |\tw|^{p-2}\tw_{\vert_{Z=0}} \;dX
\\&- \int_{\RR_+}\Big( [\dX \tw |\tw|^{p-2}\tw - \dz\psi_{\vert_{z=0}} (\dZ-\dz\psi_{\vert_{z=0}}\dX)\tw |\tw|^{p-2} \tw]_{\vert_{X=0}} \Big)\;dZ
\end{aligned}$$
in which the last term on the right-hand side vanishes due to the jump conditions \eqref{BC-jump-w} on $\tw$ and on $\dX\tw$. Thus, we obtain
$$\begin{aligned}
\int_{\RR^2_+}\Delta^{\psi_0}_{XZ} \tw \tw \;dXdZ 
=& -(p-1)\int_{\RR^2_+}|\tw|^{p-2}|\nabla^{\psi_0}_{XZ} \tw|^2 \;dXdZ - \int_{\RR}  \dZ\tw |\tw|^{p-2}\tw_{\vert_{Z=0}} \;dX.
\end{aligned}$$

\bigskip
\ind
Collecting, we have shown 
\begin{equation}\label{L2-est-tw}\begin{aligned}
\frac{d}{dt}\|\tw\|_{L^p_{XZ}}^p
&+ p(p-1) \int_{\RR^2_+} |\tw|^{p-2}|\nabla_{X,Z}^{\psi_0} \tw |^2 \; dXdZ \\&\lesssim\quad  \|\tw \|_{L^p_{XZ}}^{p-1} \|J_\pm\|_{L^p_Z}  - \int_{\RR}  \dZ\tw |\tw|^{p-2}\tw_{\vert_{Z=0}} \;dX.
\end{aligned}\end{equation}
For the boundary term, the Young's inequality yields
$$ \int_{\RR}  |\dZ\tw |\tw|^{p-2}\tw_{\vert_{Z=0}} |\;dX \quad\lesssim\quad \| \tw_{\vert_{Z=0}}\|_{L^p_X} ^p + \| \dZ \tw_{\vert_{Z=0}}\|^p_{L^p_X}.$$
The first boundary term can be easily treated by the trace inequality. We treat the second boundary term by the $H_Z^1$ energy estimate. To this end, we take $Z$-derivative of the equation \eqref{eqs-tw} and multiply by $|\dZ \tw|^{p-2} \dZ \tw$ to the resulting equation. We simply get 
$$\begin{aligned}
\frac1p\frac{d}{dt}\|\dZ\tw\|_{L^p_{XZ}}^p&= \int_{\RR^2_+}\Big(\Delta^{\psi_0}_{XZ} \dZ \tw + \dZ J_\pm e^{\mp X}\Big)|\dZ \tw|^{p-2} \dZ\tw \;dXdZ.
\end{aligned}$$
Again, by applying the H\"older's inequality to the last term on the right-hand side, we have 
$$\begin{aligned}
\int_{\RR^2_+} \dZ J_\pm e^{\mp X}|\dZ \tw|^{p-2} \dZ\tw \;dXdZ \quad\lesssim\quad \|\dZ J_\pm\|_{L^p_Z}\|\dZ \tw\|_{L^p_{XZ}}^{p-1} .
\end{aligned}$$
Next, the integration by parts yields
$$\begin{aligned}
\int_{\RR^2_+}&\Delta^{\psi_0}_{XZ} |\dZ \tw|^{p-2}\dZ\tw \dZ\tw \;dXdZ 
\\=& -(p-1)\int_{\RR^2_+} |\dZ \tw|^{p-2} |\nabla^{\psi_0}_{XZ} \dZ\tw|^2 \;dXdZ - \int_{\RR}  (\dZ-\dz\psi_{\vert_{z=0}}\dX)\dZ\tw |\dZ \tw|^{p-2}\dZ\tw_{\vert_{Z=0}} \;dX
\\&- \int_{\RR_+}\Big( [\D_{XZ}^2\tw |\dZ \tw|^{p-2}\dZ\tw - \dz\psi_{\vert_{z=0}} (\dZ-\dz\psi_{\vert_{z=0}}\dX)\dZ\tw |\dZ \tw|^{p-2}\dZ\tw]_{\vert_{X=0}} \Big)\;dZ
\end{aligned}$$
in which again the last term on the right-hand side vanishes due to the jump condition \eqref{BC-jump-w}. By using the equation for $\tw$, we can write the boundary term as
$$\begin{aligned}
- &\int_{\RR}  (\dZ-\dz\psi_{\vert_{z=0}}\dX)\dZ\tw |\dZ \tw|^{p-2}\dZ\tw_{\vert_{Z=0}} \;dX 
\\&= - \int_{\RR} \Big( (\dZ-\dz\psi_{\vert_{z=0}}\dX)^2\tw  + \dz\psi_{\vert_{z=0}} (\dZ-\dz\psi_{\vert_{z=0}}\dX)\dX \tw \Big)|\dZ \tw|^{p-2}\dZ\tw_{\vert_{Z=0}} \;dX
\\&= - \int_{\RR} \Big(\dt \tw  - (1+|\dz\psi_{\vert_{z=0}}|^2)\D_X^2 \tw - J_\pm e^{\mp X} + \dz\psi_{\vert_{z=0}} \D_{XZ}^2 \tw \Big)|\dZ \tw|^{p-2}\dZ\tw_{\vert_{Z=0}} \;dX, 
\end{aligned}$$
in which the integral term involving $\D_{XZ}^2\tw $ vanishes due to the jump condition \eqref{BC-jump-w} and the fact that it is a perfect derivative in $X$. For the other terms, we note that at $Z=0$ we have 
$$\begin{aligned}\dt \tw  - (1+|\dz\psi_{\vert_{z=0}}|^2)\D_X^2 \tw -J_\pm e^{\mp X}  &= \dt V_{KH}  - (1+|\dz\psi_{\vert_{z=0}}|^2)\D_X^2 V_{KH}  + c_\pm e^{\mp X}  
\\&= c_\pm e^{\mp X}.\end{aligned} $$
for some constant $c_\pm$; here, the last identity was due to a use of the equation for $V_{KH}$. Thus, using this and the Sobolev embedding, we have 
$$\begin{aligned}
- \int_{\RR}  (\dZ-\dz\psi_{\vert_{z=0}}\dX) & \dZ\tw |\dZ \tw|^{p-2}\dZ\tw_{\vert_{Z=0}} \;dX 
\\&\quad\lesssim\quad \| \dZ \tw_{\vert_{Z=0}}\|_{L^p_X}
\\&\quad\lesssim\quad \|\dZ \tw \|_{L^p_{XZ}} + \Big(\int_{\RR^2_+} |\dZ \tw|^{p-2} |\D_Z^2 \tw |^2 \; dZdX\Big)^{1/p}.\end{aligned}$$
Thus, applying the Young's inequality to the last term and combing all the above estimates, we obtain 
$$\begin{aligned}
\frac{d}{dt}\|\dZ\tw\|_{L^p_{XZ}}^p&+ \int_{\RR^2_+}|\dZ \tw |^{p-2} |\nabla_{X,Z}^{\psi_0} \dZ\tw|^2 \; dXdZ
\\&\lesssim\quad 1 +  \|\dZ \tw \|_{L^p_{XZ}} +  \|\dZ J_\pm\|_{L^p_Z}\|\dZ \tw\|_{L^p_{XZ}}^{p-1} .
\end{aligned}$$
This together with the $L^2$ estimate \eqref{L2-est-tw} yields the lemma at once.
\end{proof}

\bigskip
\ind
We also obtain the following $X$-derivative estimates.
\begin{lemma}\label{lem-box-layer-x} For any solutions $\tw$ to \eqref{eqs-tw}--\eqref{BC-bdry-w}, there holds
\begin{equation*}
\begin{aligned}
\frac {d}{dt} \|\dX \tw \|_{L^p_XW^{1,p}_Z}^p&+\int_{\RR^2_+}| \dX \tw |^{p-2} |\nabla_{X,Z}^{\psi_0}\dX \tw|^2 \; dXdZ+\int_{\RR^2_+}|\D_{XZ}^2\tw |^{p-2} |\nabla_{X,Z}^{\psi_0} \D_{XZ}^2\tw|^2 \; dXdZ
\\& \lesssim\quad  1+ \|J_\pm\|_{W^{1,p}_Z} ^p + \| \tw \|_{W^{1,p}_{XZ}}^p .
\end{aligned}\end{equation*}
\end{lemma}
\begin{proof} The proof of this lemma follows word by word of that of the Lemma \ref{lem-box-layer}, upon noting that the jump of discontinuity of $\partial_X^2 \tw $ across $\{X=0\}$ can be computed through the equation \eqref{eqs-tw} for $\tw$ to give $[\D_X^2 \tw]_{\vert_{X=0}} = J_+ - J_-$. Note also that we may need to apply the Sobolev embedding:
$$\| \dZ \tw_{\vert_{X=0}}\|_{L^p_Z} \quad\lesssim\quad\|\dZ \tw \|_{L^p_{XZ}} + \Big(\int_{\RR^2_+} |\dZ \tw|^{p-2} |\D_{XZ}^2 \tw |^2 \; dZdX\Big)^{1/p}.$$
We thus omit the further detail of the proof of the lemma. 
\end{proof}

To conclude this subsection, we summarize our estimate for $V_b$ in the following lemma.
\begin{lemma}\label{lem-layer-Vb} There exists a unique solution $V_b$ of the problem \eqref{eqs-box1} with the boundary and jump conditions \eqref{jump-Vb} and initial data \eqref{initial-Vb}. Furthermore, for any $p>1$, there exists some positive constant $C_0$ that depends on $p$ and $T$ such that there holds
\begin{equation}\label{est-Vb-layer} \begin{aligned}
\sup_{0\le t\le T} &\|V_b(t)\|^p_{W^{1,p}_{XZ}}  + \int_0^T \int_{\RR^2_+} |\D_X^2V_b |^p \; dXdZ dt
\\&\le C_0\Big(|[v_0(x,0)]_{\vert_{x=0}}|^p+ \| [v_0(x,\cdot)]_{\vert_{x=0}} \|^p_{W^{1,p}(\RR_+)} +   \| v_{ 0} (\cdot,0)  \|^p_{W^{1,p} (\RR_\pm)}\Big).
\end{aligned}
\end{equation}
\end{lemma}
\begin{proof} This is a collection of estimates from Lemmas \ref{lem-box-layer} and \ref{lem-box-layer-x}, the estimate \eqref{est-Jpm} on $J_\pm$, the jump estimate \eqref{est-PrVp-jump} from Lemma \ref{lem-Vp}, and a use of the standard Gronwall inequality. Indeed, Lemmas \ref{lem-box-layer} and \ref{lem-box-layer-x} inparticular yields
$$\sup_{0\le t \le T} \| \tw \|_{W^{1,p}_{XZ}}^p + \int_0^T\int_{\RR^2_+} |\dX \tw |^{p-2} |\D_X^2 \tw |^2 \; dXdZ dt$$
is bounded. This together with the standard Young's inequality yields that
$$\begin{aligned} \int_0^T \int_{\RR^2_+} |\D_X^2 \tw |^p \; dXdZ dt  &= \int_0^T \int_{\RR^2_+} |\dX \tw|^{\frac{2-p}{2}p} |\dX \tw|^{\frac{p-2}{2}p}|\D_X^2 \tw|^p \; dXdZ dt 
\\&\le \int_0^T \int_{\RR^2_+} |\dX \tw|^{p} \; dXdZ dt + \int_0^T \int_{\RR^2_+} |\dX \tw |^{p-2}|\D_X^2 \tw |^2 \; dXdZ dt
\end{aligned}
$$
is also bounded. The lemma is proved. 
\end{proof}

\subsection{Remainders}

We observe that, with the above profiles, the functions   $(u^\e_{app}, v^\e_{app})$ given by the formula \eqref{appx-solns} satisfy
\begin{equation}\label{plane-NSE-Approx} 
\begin{aligned}\dt u^\e_{app} &= \e \dz^2 u^\e_{app} + E^u ,
\\ \dt v^\e_{app}  + (u^\e_{app}  - u_0) \dx v^\e_{app}  &= \e  \Delta^\psi_{xz} v^\e_{app}  + E^v ,
\\ (u^\e_{app} ,v^\e_{app} )_{\vert_{z=0}} &=0 ,
\\ [ v^\e_{app}] _{\vert_{x = 0}}  &=0 \\  [\dx  v^\e_{app} ] _{\vert_{x = 0}} &=0 ,
\end{aligned}
\end{equation}
where direct computations give
$$E^u=\e \dz^2 u_0,$$
and
$$ \begin{aligned}
E^v&= \e \Delta_{xz}^\psi v_0 + (\Delta_{XZ}^\psi - \Delta_{XZ}^{\psi_0}) V_b - 2\sqrt \e\dz\psi \D_{zX}^2 V_{KH}- \sqrt \e\dz^2 \psi\dX V_{KH} + \e \dz^2V_{KH}\\
&\qquad +\e (1+|\dz\psi|^2))\dx^2 V_P- 2\sqrt\e\dz\psi \D_{xZ}^2V_P - \e \dz^2 \psi\dx V_P - U_p \dx v^\e_{app} .
\end{aligned}$$
\begin{remark}[] 
\label{notL2}
We note that $E^v$ contains  a singular term: $\frac{1}{\sqrt \e}U_{p}(t,z/\sqrt \e)\dX V_{KH}(t,x/\sqrt \e,z)$ in $U_p\dx v_{app}^\e$. This singular term is a-priori not better than bounded in $L^2$ and thus we can't obtain the convergence in the $L^2$ space this way. However, its $L^p$ norm has an order of $\e^{1/p-1/2}$, which tends to zero in $L^p$, for $1< p<2$, as $\e \to 0$. 
\end{remark}
\bigskip
\ind

\begin{proposition}\label{prop-estE} For all $p> 1$, there hold uniform estimates:
\begin{equation}\label{eqs-estE}
\begin{aligned}
\|E^u(t)\|_{L^2_z} &\quad\le\quad \e \|\dz^2 u_0 \|_{L^2_z},\qquad
\int_0^T \|E^v(t)\|^p_{L^p_{xz}} \; dt\quad\le\quad C_{in}(p,T) \e^{1-p/2},
\end{aligned}
\end{equation}
for some positive constant $C_{in}(p,T)$ that depends continuously on the initial data, the discontinuous jump of the Euler flow $v_0$, the number $p>1$, and the time $T$. More precisely, the constant $C_{in}(p,T)$ is bounded by $$ C_0\Big(|[v_0(x,0)]_{\vert_{x=0}}|^p+ \| [v_0(x,\cdot)]_{\vert_{x=0}} \|^p_{W^{1,p}(\RR_+)} +   \| v_{ 0} (\cdot,0)  \|^p_{W^{1,p} (\RR_\pm)}+ \|\Delta v_0\|_{L^p_{xz} (\RR_\pm\times \RR_+)}^p\Big)$$
for some $C_0$ that depends only on $T$ and $p$. 
\end{proposition}

We now give a proof of the Proposition \ref{prop-estE}. The first estimate is clear from the definition $E^u = \e \dz^2 u_0$. We prove the second estimate. We will use the following simple lemma.
\begin{lemma}\label{lem-Zz} For any reasonable function $u = u(z)$, there holds
\begin{equation*} \| u(z/\sqrt \e)\|_{L_z^2(\RR_+)} \quad\lesssim\quad \e^{1/4} \|u(Z)\|_{L_Z^2(\RR_+)}.\end{equation*}
\end{lemma}
\begin{proof} It is clear by changing of variable from $z$ to $z/\sqrt\e$. \end{proof}

We now check term by term in $E^v$. The term $\e \Delta^\psi_{xz} v_0$ is clear, giving the contribution of $\e \|\Delta v_0\|_{L^p_{xz} (\RR_\pm\times \RR_+)}^p$. Next, note that 
$$ \Delta^\psi_{XZ} - \Delta_{XZ}^{\psi_0} =(|\dz\psi|^2 - |\dz\psi_{\vert_{z=0}}|^2)\partial_X^2- 2(\dz\psi - \dz \psi_{\vert_{z=0}} )\partial_{XZ}^2 - \dz^2 \psi\dX.$$
Thus, the estimate \eqref{est-Vb-layer}  for $V_b$ precisely gives us the desired $L^p$ estimate for $(\Delta^\psi_{XZ} - \Delta_{XZ}^{\psi_0})V_b$, after a change of variables $(x,z)$ to $(X,Z)$ with $X = x/\sqrt \e, Z = z/\sqrt \e$, yielding a small factor of $\e$. Similarly, for all the terms:
$$- 2\sqrt \e\dz\psi \D_{zX}^2 V_{KH}- \sqrt \e\dz^2 \psi\dX V_{KH} + \e \dz^2V_{KH} + \e (1+|\dz\psi|^2))\dx^2 V_P- 2\sqrt\e\dz\psi \D_{xZ}^2V_P - \e \dz^2 \psi\dx V_P$$
the estimates from Lemmas \ref{lem-Up}, \ref{lem-Vp}, and \ref{lem-VKH} immediately yield that the $L^p$ norm of these are bounded by 
$$ C_0\e^{1/2}\Big(|[v_0(x,0)]_{\vert_{x=0}}|+ \| [v_0(x,\cdot)]_{\vert_{x=0}} \|_{W^{1,p}(\RR_+)} +   \| v_{ 0} (\cdot,0)  \|_{W^{1,p} (\RR_\pm)}\Big).$$

Finally, let us treat the term $U_p \dx v^\e_{app}$. From the definition of $v^\e_{app}$, the singular terms in $U_p \dx v^\e_{app}$ are 
$$\frac{1}{\sqrt \e} U_p(t,Z) \dX V_{KH}(t,X,z) + \frac 1{\sqrt \e}U_p(t,Z) \dX V_b (t,X,Z).$$
We then use the Lemma \ref{lem-Zz} to treat these singular terms. For example, we compute 
$$\begin{aligned}
\e^{-p/2}\int_{\RR^2_+} |U_p (t,Z) \dX V_b(t,X,Z)|^p \; dx dz &\quad\lesssim\quad \e^{1-p/2}\int_{\RR^2_+} |U_p (t,Z) \dX V_b(t,X,Z)|^p \; dXdZ 
\\& \quad\lesssim\quad \e^{1-p/2} \|U_p\|_{L^\infty_Z}^p \|\dX V_b\|_{L^p_{XZ}}^p\\&\quad\lesssim\quad \e^{1-p/2}.
\end{aligned}$$ 
Other terms are entirely similar. This completes the proof of the estimate \eqref{eqs-estE}, and thus the Proposition \ref{prop-estE}. 

\subsection{Convergence}

We are ready to prove the convergence stated in the main theorem. 

Now we consider the solutions $R^u,R^v$ of the following problem:
\begin{equation}\label{eqs-remainders} 
\begin{aligned}\dt R^u &= \e \dz^2 R^u + E^u, \\
\dt R^v + (U_p + R^u) \dx R^v &=\e\Delta^\psi_{xz} R^v -R^u\dx v^\e_{app} + E^v \\
(R^u,R^v)_{\vert_{z=0}} &=0,\\\lim_{z\to +\infty}(R^u,R^v) &= 0, \\
[R^v]_{\vert_{x=0}} = [\dx R^v]_{x=0} &=0 , \\ 
(R^u,R^v)_{\vert_{t=0}} &=0 .
\end{aligned}\end{equation}
Then the functions  $(u^\e,v^\e)$  defined by 
$$\begin{aligned} u^\e(t,z) &= u^\e_{app}(t,z) + R^u(t,z)\\
v^\e(t,x,z) &= v^\e_{app}(t,x,z) + R^v(t,x,z) ,
\end{aligned}$$
satisfy the equations \eqref{plane-NSE}-\eqref{plane-NScds2}-\eqref{nojump3}.

The  well-posedness of the problem \eqref{eqs-remainders}  follows at once from the following a-priori estimates:
\begin{lemma}\label{lem-est-Ruv} There hold
$$\begin{aligned} \frac{d}{dt} \|R^u\|_{L^2_z}^2+ \e \|\dz R^u\|^2_{L^2_z}& \quad\lesssim\quad \|R^u\|_{L^2_z}\|E^u\|_{L^2_z},
\\
\frac{d}{dt} \|R^v\|^p_{L^p_{x,z}}+\e\int_{\RR^2_+}|R^v|^{p-2}|\nabla_{x,z}^\psi R^v|^2 \; dxdz 
&\quad\lesssim\quad\Big(\|R^u \dx v^\e_{app}\|_{L^p_{xz}}+ \|E^v\|_{L^p_{xz}}\Big) \|R^v\|_{L^p_{xz}}^{p-1},
\end{aligned}$$
where  $\nabla^\psi_{x,z}:=(\dx,\dz-\dz\psi\dx)$. 
\end{lemma}
\begin{proof} Multiply by $R^u$ and $|R^v|^{p-2}R^v$  the respective equations in \eqref{eqs-remainders} and integrate the resulting equations over $\RR_+$ or $\RR^2_+$. The claimed estimate for $R^u$ is straightforward. For the $R^v$ estimate, we have 
$$\begin{aligned}\frac 1p\frac{d}{dt} \int_{\RR_+^2}|R^v|^p &+ \int_{\RR_+^2}\Big( U_p+ R^u)  |R^{v}|^{p-2}\dx R^v R^v\;dxdz\\
&= \int_{\RR_+^2}\Big(\e\Delta^\psi_{xz} R^v -R^u\dx v^\e_{app} - E^v\Big)|R^v|^{p-2}R^v \; dxdz.\end{aligned}$$
We first note that integration by parts yields
$$ \begin{aligned}\int_{\RR^2_+}  (U_p + R^u) |R^v|^{p-2} \dx R^v R^v  \;dxdz&= \int_{\RR^2_+}  (U_p + R^u)  \dx \Big(\frac{|R^v|^p}{p}\Big)  \;dxdz = 0
\end{aligned}$$
and 
$$\begin{aligned}  \e\int_{\RR^2_+}\Delta^\psi_{xz} R^v |R^v|^{p-2}R^v \;dxdz &= - \e\int_{\RR^2_+}|R^v|^{p-2}|\nabla_{x,z}^\psi R^v|^2 \; dxdz, 
\end{aligned}$$
upon noting that there is no contribution on the boundary and the interface due to the vanishing boundary and jump conditions for $R^v$.
 Finally, the standard H\"older inequality yields 
$$ \begin{aligned} 
\int_{\RR_+^2} (|R^u \dx v^\e_{app}|+ |E^v|)|R^v|^{p-1} \; dxdz\quad\lesssim\quad \Big(\|R^u \dx v^\e_{app}\|_{L^p_{xz}}+ \|E^v\|_{L^p_{xz}}\Big) \|R^v\|_{L^p_{xz}}^{p-1}.
\end{aligned}$$

\ind
Collecting these estimates together proves the lemma.
\end{proof}

\bigskip
\ind

It is straightforward to verify that 
\begin{equation*}\begin{aligned}
(u^\e_{app} (t,z) , v^\e_{app} (t,x,z) ) & \to \vv^0 \qquad \mbox{in } L^\infty(0,T;  L^{2} (\RR_+)   \times L^{p} ( \RR \times \RR_+)) ,
\end{aligned}
\end{equation*} 
as $\e \to 0$, upon using the estimates on the profiles from Lemmas \ref{lem-Up}, \ref{lem-Vp}, \ref{lem-VKH}, and \ref{lem-layer-Vb}, and the Lemma \ref{lem-Zz}.
Therefore in order to prove Theorem \ref{theo-conv}, it remains to prove that 
\begin{equation}\label{eqs-convApp}\begin{aligned}
(R^u (t,z) , R^v (t,x,z) ) & \to 0 \qquad \mbox{in } L^\infty(0,T;  L^{2} (\RR_+)   \times L^{p} ( \RR \times \RR_+)) ,
\end{aligned}
\end{equation} 
as $\e \to 0$.

From Lemma \ref{lem-est-Ruv} and Proposition \ref{prop-estE}, by the standard ODE estimate and the Gronwall inequality, we immediately obtain uniform bounds
\begin{equation*}\|R^u\|_{L^\infty_tL^2_z} + \e^{1/2} \|\dz R^u\|_{L^2_tL^2_z} \quad\lesssim\quad \e,\end{equation*}
with noting that $\|E^u\|_{L^2_tL^2_z} \lesssim \e$ and $R^u_{\vert_{t=0}}=0$. In addition, this estimate yields
$$\int_0^T \|R^u(t,\cdot)\|_{L^\infty_z}^4 dt \quad\lesssim\quad \|R^u\|^2_{L^\infty_t L^2_z} \|\dz R^u\|_{L^2_tL^2_z}^2\quad\lesssim\quad \e^3.$$
This together with the bound $\|\dx v^\e_{app}\|^p_{L^\infty_tL^p_{x,z}}\lesssim 1$, which again follows from the estimates on the profiles, yields
$$ \int_0^T \|R^u\dx v^\e_{app}\|^p_{L^p_{x,z}} \quad\lesssim\quad \|\dx v^\e_{app}\|^p_{L^\infty_tL^p_{x,z}}  \int_0^T \|R^u \|_{L^\infty_z}^p \;dt\lesssim \e^{3p/4} .$$
In addition, the second estimate from Lemma \ref{lem-est-Ruv} implies
$$\frac{d}{dt} \|R^v\|^p_{L^p_{xz}}  \quad\lesssim\quad  \|R^v\|_{L^p_{xz}}^{p} + \Big(\|R^u \dx v^\e_{app}\|^p_{L^p_{xz}}+ \|E^v\|^p_{L^p_{xz}}\Big),$$
which gives
$$ \|R^v(t)\|^p_{L^p_{xz}}  \quad\lesssim\quad \int_0^t \|R^v(s)\|_{L^p_{xz}}^{p} ds + \Big(\e^{3/4}+ \e^{1-p/2}\Big).$$
The Gronwall inequality then yields
\begin{equation*} \|R^v\|^p_{L^\infty_tL^p_{xz}}  \quad\lesssim\quad \e^{3/4}+ \e^{1-p/2},\end{equation*}
which tends to zero as $\e\to 0$, for $p<2$. 

This ends the proof of the convergence \eqref{eqs-convApp}, and thus of  Theorem \ref{theo-conv}.

\bigskip
\ind
{\bf Acknowledgements.} The research of T. N. was supported in part by the Foundation Sciences Math\'ematiques de Paris through a 2009-2010 post-doctoral fellowship and the National Science Foundation through the grant DMS-1108821.
\bigskip

The second author was partially supported by  the Lefschetz Center for Dynamical Systems at Brown University and the Agence Nationale de la Recherche, Project CISIFS,  grant ANR-09-BLAN-0213-02. He also warmly thanks  the Division of Applied Mathematics  at Brown University for their kind hospitality during his visit in April $2011$.

\end{document}